\documentclass[12pt]{article}
\usepackage[a4paper,left=15mm,top=20mm,right=20mm,bottom=25mm]{geometry}

\usepackage{amsmath, amsfonts, amssymb, amsthm}
\usepackage[numbers,sort&compress]{natbib}
\usepackage{authblk}
\usepackage{algorithm}
\usepackage[noend]{algpseudocode}
\usepackage{graphicx}  
\usepackage{mathtools}
\usepackage{subfig}
\usepackage{mathrsfs}
\usepackage{mathtools}

\usepackage[mathcal]{euscript}
\usepackage{mathptmx} 
\usepackage{xspace}

\usepackage[colorlinks]{hyperref}
\hypersetup{
	citecolor=blue,
   linkcolor=blue,
}

\usepackage[font=footnotesize,width=.85\textwidth,labelfont=bf]{caption}

\usepackage{amsmath, amsfonts, amssymb}
\usepackage[numbers,sort&compress]{natbib}
\usepackage{authblk}
\usepackage{graphicx}  
\usepackage{mathtools}
\usepackage{subfig}
\usepackage{mathrsfs}

\usepackage[mathcal]{euscript}
\usepackage{mathptmx} 
\usepackage{xspace}

\usepackage[colorlinks]{hyperref}
\hypersetup{
	citecolor=blue,
   linkcolor=blue,
}

\newcommand{\mc}{\mathcal}
\newcommand{\ms}{\mathscr}
\newcommand{\HC}{\textsf{HLC}\xspace}
\newcommand{\IC}{\textsf{IC}\xspace}
\newcommand{\lc}[1]{\ensuremath{_{{\scriptscriptstyle #1}}}}

\newcommand{\Ns}{\ensuremath{\mc{N}}\xspace}
\newcommand{\Cls}{\ensuremath{\mc{C}}\xspace}
\newcommand{\Cl}[1]{\ensuremath{C_{\scriptscriptstyle #1}}\xspace}
\newcommand{\rootT}[1]{\ensuremath{\rho_{_{\scriptscriptstyle #1}}}}
\newcommand{\To}[1]{\ensuremath{\ms{T}({#1})}\xspace}
\newcommand{\irr}[1]{\ensuremath{[#1]_{\textrm{irr}}}\xspace}
\newcommand{\PS}[1]{\ensuremath{\mc{P}(#1)}\xspace}

\DeclareMathOperator{\lca}{lca}
\newcommand{\rel}{\mathcal{X}}
\newcommand{\parent}{\mathsf{par}}

\providecommand{\keywords}[1]{\textbf{\textit{Keywords: }} #1}

\newtheorem{theorem}{Theorem}
\newtheorem{lemma}{Lemma}
\newtheorem{proposition}{Proposition}
\newtheorem{corollary}{Corollary}

\newtheorem{definition}{Definition}

\title{Alternative Characterizations of Fitch's Xenology Relation}

\author[1,2,*]{Marc Hellmuth}
\author[3,4,*]{Carsten R.\ Seemann}

\affil[1]{Institute	 of Mathematics and Computer Science, University of Greifswald, Walther-Rathenau-Strasse 47, D-17487 Greifswald, Germany }

\affil[2]{
	Saarland University, Center for Bioinformatics, Building E 2.1, P.O.\ Box 151150, D-66041 Saarbr{\"u}cken, Germany
	  }

\affil[3]{ 
	Bioinformatics Group, Department of Computer Science, Leipzig University, H{\"a}rtelstra{\ss}e 16-18,  D-04107 Leipzig, Germany\\
}
	
\affil[4]{ 
	Max-Planck-Institute for Mathematics in the Sciences, Inselstra{\ss}e 22,  D-04103 Leipzig, Germany
	}

\affil[*]{Email: \texttt{mhellmuth@mailbox.org} (MH), \texttt{carsten@bioinf.uni-leipzig.de} (CS)	}

\date{}

\begin{document}


\maketitle

\abstract{
According to Walter M.\ Fitch, two genes are xenologs if they are separated by at least one horizontal gene transfer.
This concept is formalized through Fitch relations, which are defined as
binary relations that comprise all pairs $(x,y)$ of genes $x$ and $y$ for which $y$
has been horizontally transferred at least once since it diverged from the
least common ancestor of $x$ and $y$. This definition, in particular, 
preserves the directional character of the transfer. 
Fitch relations are characterized by a small set of
forbidden induced subgraphs on three vertices and can be recognized in linear
time. 

In this contribution, we provide two novel characterizations of Fitch relations and
present an alternative, short and elegant proof of the characterization theorem
established by Gei{\ss} et al.\ in  
\emph{J.\ Math.\ Bio 77(5), 2018}.
 }

\bigskip
\noindent 

\sloppy
\section{Introduction}

Genes are the molecular units of heredity holding the information to build
and maintain cells. A gene family covers all homologous genes, that is,
genes that share a common ancestor. During evolution, genes are mutated,
duplicated, lost and passed to organisms through speciation or horizontal
gene transfer (HGT), which is the exchange of genetic material among
co-existing species. Therefore, the history of a gene family is
equivalently described by a vertex- and edge-labeled rooted phylogenetic
tree, called (event-labeled) gene tree, in which the leaves correspond to
extant genes and the internal vertices to ancestral genes. The label of a
vertex highlights the event at the origin of the divergence leading to the
offspring, namely speciation and duplication, while edge-labels express
whether an edge corresponds to HGT or not \cite{Fitch:00}.

Homology relations are binary relations between homologous genes and are defined
by the particular vertex- and edge-labels of the gene tree
\cite{Fitch:70,Fitch:00,Koonin:05,Jensen:01}. Prominent examples of homology relations
are the orthology and paralogy relation that contain all pairs of genes $(x,y)$
where the last common ancestor is a speciation and duplication event, respectively
\cite{Altenhoff:16,Jensen:01,lafond2015orthology,Lafond:13,DEML:16,Lafond:16,Hellmuth:13a,Boecker:98,hellmuth_phylogenomics_2015}, 
RGC-relations that capture the structure of rare genomic changes (RGCs)
\cite{Hellmuth:17a,Boore:06,Donath:14a,Dutilh:08,Rokas:00} 
and the xenology relation that is defined in terms of HGT
\cite{Fitch:00,Hellmuth:17,GHLS:17,Geiss2018,Hellmuth:16a,Dessimoz2008}. 

Although, homology relations are defined by the \emph{true} evolutionary history
of the genes, which is usually not known with confidence, there are many methods
that allow to infer certain types of homology relations directly from genomic
sequence data without requiring any \emph{a priori} knowledge about the topology of
either the gene or the species tree. This includes tools for estimating and
resolving orthology assignments at the level of gene pairs, and thus, to derive 
orthology relations (see e.g.\
\cite{altenhoff2017oma,Lechner:11a,Lechner:14,WPFR:07,Mahmood30122011,inparanoid:10});
methods to infer RGC-relations (see \cite{Hellmuth:17a} for an overview) or
methods to infer HGT
by using certain characteristics of the genome sequences (see
\cite{LO:02,RSLD15,Rancurel:17} for an overview). While best match heuristics
have been very successful as approximations of the orthology relation
\cite{Altenhoff:16, Nichio:17}, the inference of HGT is still challenging. In
particular, the inference of pairwise xenology relationships has not been
satisfactorily solved. It is, however, not at all a hopeless task, since genes
that are imported by HGT from an ancestor of species A into an ancestor of
species B are expected to be more closely related than one could expect 
from the bulk of the genome \cite{RSLD15, Novichkov:04,Dessimoz2008}.

Homology relations are of fundamental importance in many fields of mathematical
and computational biology including the reconstruction of evolutionary
relationships across species
\cite{GCMRM:79,DBH-survey05,hellmuth_phylogenomics_2015,HW:16b,Dutilh:08,Boore:98,Boore:06,Rokas:00},
functional genomics and gene organization in species
\cite{Koonin:05,GK13,TKL:97,TG+00,Sempere:06,Rogozin:05,Sankoff:82,Lavrov:07,Donath:14a},
and the identification and testing of proposed mechanisms of genome evolution
\cite{MK:96,TKL:97,Abascal:12}. It is therefore of central interest to
understand whether such \emph{inferred} relations are ``mathematically correct'', i.e.,
whether there is an event-labeled gene tree that can explain the
given relations and thus, provides some evidence about the inferred data. 

By way of example, a mathematically correct orthology relation 
must form a co-graph (graphs that do not contain induced paths on four vertices) \cite{Corneil:81}
and is associated to a unique co-tree, which is equivalent
to a not necessarily fully resolved event-labeled gene tree that explains the
given orthology relation \cite{Boecker:98,Hellmuth:13a}. Empirically estimated
orthology relationships in general violate the co-graph property, suggesting
co-graph editing as a means to correct the initial estimate
\cite{lafond2015orthology,Lafond:13,DEML:16,Lafond:16,Hellmuth:13a,HW:16b,hellmuth_phylogenomics_2015}.
On genome-wide data sets, the cographs and their co-trees can then be used to even  infer the
evolutionary history of the underlying species
\cite{HHH+12,Lafond:16,HW:16b,Hellmuth:17,hellmuth_phylogenomics_2015}. Hence, 
understanding the mathematical structure of orthology helped to obtain deeper
insights into the complex processes that drive molecular evolution.

This contribution is concerned with the mathematical structure of xenology, which
is intimately related to HGT. The horizontal transfer of genomic material can be
annotated in the gene tree by assigning a label to the edge that points from
the horizontal transfer event to the next event in the history of the copy. 
The concept of xenology, as introduced by Walter M.\ Fitch \cite{Fitch:00}, 
calls two homologous genes \emph{xenologs}, if their history, 
since their common ancestor, involves a horizontal transfer
for at least one of them \cite{Jensen:01,Fitch:00}.
In other
words, two genes $x$ and $y$ are xenologs if the unique path between $x$ and $y$
in the underlying event-labeled gene tree contains a transfer edge. The class of
such relations has been characterized by Hellmuth et al.\ \cite{GHLS:17} 
and coincides with the class of complete multipartite graphs. 
Note that HGT is	
intrinsically a directional event, i.e., there is a clear distinction between
the horizontally transferred ``copy'' and the ``original'' that continues to be
vertically transferred. Preserving the directionality of horizontal transfer,
Gei{\ss} et al.\ \cite{Geiss2018} formalized this concept and introduced
\emph{Fitch relations}, which comprise all pairs of genes $(x,y)$ for which the
unique path from the last common ancestor $\lca(x,y)$ to $y$ in the gene tree
contains a transfer edge. It has been shown by Gei{\ss} et al.\ \cite{Geiss2018}
that Fitch relations are characterized by the absence of eight forbidden
subgraphs on three vertices and can be recognized in linear time. 

In order to understand Fitch relations in more detail, 
we provide in this contribution two additional characterizations, 
the first one is based on neighborhoods and the second one is based on forbidden subrelations on three vertices. 
The proof of the characterization theorem in \cite[Thm.\ 2]{Geiss2018} is quite involved and very 
technical, and it includes plenty of case studies. We will use the new characterization
to provide a novel, simpler and elegant proof of this theorem.

\section{Preliminaries}

\paragraph{\bf Basics.}
For a finite set $L$ we put $\irr{L \times L} \coloneqq (L \times L) \setminus \{ (x,x) \colon x \in L\}$ and 
${L \choose k} \coloneqq \{ L' \subseteq L \colon |L'| = k\}$.

All binary relations considered here are irreflexive, and we 
omit to mention it each time and simply call them \emph{relations}. 
A relation $\rel$ \emph{on $L$} is a subset $\rel \subseteq \irr{L\times L}$.

We consider directed graphs  $G=(V,E)$ with finite vertex set $V$ and 
edge set $E\subseteq \irr{V\times V}$ and thus, $G$ does not contain loops or multiple edges. 
For a directed graph $G=(V,E)$ and a subset $W\subseteq V$ 
let $G[W] = (W,F)$ denote the \emph{induced subgraph} of $G$ that has edge set $F\subseteq E$ such that 
every edge $(x,y)\in E$ with $x,y\in W$
is also contained in $F$. 

Every relation $\rel$ on $L$ can be represented by a directed graph 
$G=(V,E)$ with vertex set $V=L$ and edge set $E=\rel$. In what follows, we therefore will
interchangeably speak of $\rel$ as graph or relation and use
the standard graph terminology such as ``induced subgraph of 
$\rel$''.

\paragraph{\bf  Trees.}
A \emph{rooted tree} $T=(V,E)$ (on $L$) is 
an undirected connected cycle-free graph with finite leaf set $L$
and one distinguished  vertex $\rootT{T}$ that is called the \emph{root of T}.
The set of {\em inner} vertices of $T$ is denoted by $V^0 \coloneqq V\setminus L$. 
An edge $(v,w)$ is called \emph{inner edge} if $v,w\in V^0$ and 
\emph{outer edge} otherwise. 

In what follows, we always consider \emph{phylogenetic} trees $T$ (on $L$), that is,
rooted trees on $L$ such that the root $\rootT{T}$ has at least
degree $2$ and every other inner vertex $v\in V^0\setminus\{\rootT{T}\}$ has
at least degree $3$. Note that  phylogenetic trees $T$ on $L$ always 
satisfy $|L|\geq 2$, since $\rootT{T}$ has at least degree $2$.

Given a rooted tree $T=(V,E)$, we call $u\in V$ an \emph{ancestor} of 
$v\in V$, $u\preceq\lc{T} v$,  if $u$ lies on the unique path
from $\rootT{T}$ to $v$. We write $u\prec\lc{T} v$ for $u\preceq\lc{T} v$ and $u\neq v$. 
If neither $u\preceq\lc{T} v$ nor $v\prec\lc{T} u$, the vertices $u$ and $v$ are 
\emph{incomparable} and \emph{comparable} otherwise.  
We always write $(v,w)\in E$ to indicate that $v\prec\lc{T} w$. In the latter case, 
the unique vertex $v$ is called \emph{parent} of $w$, denoted by $\parent(w)$.
For a non-empty subset $Y\subseteq V$ of vertices, the \emph{last common
ancestor of} $Y$, denoted by $\lca\lc{T}(Y)$, is the unique $\preceq\lc{T}$-maximal
vertex of $T$ that is an ancestor of every vertex in $Y$. We will make use of
the simplified notation $\lca\lc{T}(x,y)\coloneqq\lca\lc{T}(\{x,y\})$ for $Y=\{x,y\}$.  
We will omit the explicit reference to $T$ for $\prec\lc{T}$ and $\lca\lc{T}$
 whenever it is clear which tree is considered.

\paragraph{\bf Clusters and Hierarchies.}
An arbitrary subset $\mc H \subseteq \PS{L}$ of the powerset of a finite set $L$ 
that satisfies $P\cap Q\in \{P, Q, \emptyset\}$ for all $P, Q\in \mc H$
is called \emph{hierarchy-like}. 
A \emph{hierarchy} on a finite set $L$ is a subset $\mc H \subseteq \PS{L}$
that is hierarchy-like and additionally satisfies 
$L\in \mc H$ and $\{x\}\in \mc H$ for all $x\in L$. 

Given a phylogenetic tree $T=(V,E)$, we can define for each vertex $v\in V$ 
the set of descendant leaves as 
$\Cl{T}(v) \coloneqq \{ x \in L \colon v \preceq x \}$, called a \emph{cluster} of $T$.
A cluster $\Cl{T}(v)$ is \emph{trivial} if $\Cl{T}(v) = L$ or $\Cl{T}(v)=\{v\}$ and 
\emph{non-trivial} otherwise.  
The \emph{cluster set of $T$} is then $\Cls(T) \coloneqq \{ \Cl{T}(v)\colon v \in V \}$. 
It is well-known that $\Cls(T)$ forms a hierarchy and 
that there is a one-to-one correspondence between
(isomorphism classes of) rooted trees and their cluster sets \cite{Semple2003,steel_phylogeny:_2016}, 
as summarized as follows.

\begin{theorem}[{\cite[Thm.\ 3.5.2]{Semple2003}}]
	For a given subset $\mc H \subseteq \PS{L}$, there is a 
	phylogenetic tree $T$ on $L$ with $\mc H = \Cls(T)$
	if and only if $\mc H$ is a hierarchy on $L$.

	Moreover, if there is such a phylogenetic tree $T$ on $L$, then, 
	up to isomorphism, $T$ is unique. 
  \label{thm:hierarchy-tree}
\end{theorem}

We say that a phylogenetic tree $T'$ is a \emph{coarsement} of a phylogenetic tree $T$, 
in symbols $T'< T$, if $\mathcal{C}(T')\subsetneq \mathcal{C}(T)$.

\paragraph{\bf  The Fitch Relation.}
\label{sec:fitch-graph}
We follow the notation in \cite{Geiss2018} and consider edge-labeled trees as defined as follows. 

\begin{definition}
An \emph{edge-labeled tree $(T,\lambda)$ (on $L$)} is  
a phylogenetic tree $T=(V,E)$ on $L$ together with a map $\lambda\colon E\to\{0,1\}$, 
called \emph{edge-labeling.}

For simplicity, we will speak of \emph{0-edges} and \emph{1-edges} of $T$ 
depending on their labeling.
\end{definition}

The concept of xenologs as defined by Fitch \cite{Fitch:00} 
was refined and formalized by Gei{\ss} et al.\ \cite{Geiss2018}
to preserve the directional character of gene transfer. 

\begin{definition}
  Given an edge-labeled 
  tree $(T,\lambda)$ on $L$ we set
  $(x,y)\in\rel_{(T,\lambda)}$ 
  for distinct $x,y\in L$ whenever 
	the (unique) path from  $\lca\lc{T}(x,y)$ to $y$ contains at least one 1-edge. 
\end{definition}

By construction, $\rel_{(T,\lambda)}$ is a binary irreflexive relation on $L$; 
and therefore, it can be regarded as a directed graph. 
It is easy to check that $\rel_{(T,\lambda)}$ is in
general neither symmetric nor antisymmetric. See Fig.\ \ref{fig:exmpl} for
an illustrative example. 

\begin{figure}[t]
\begin{center}
  \includegraphics[width=0.85\textwidth]{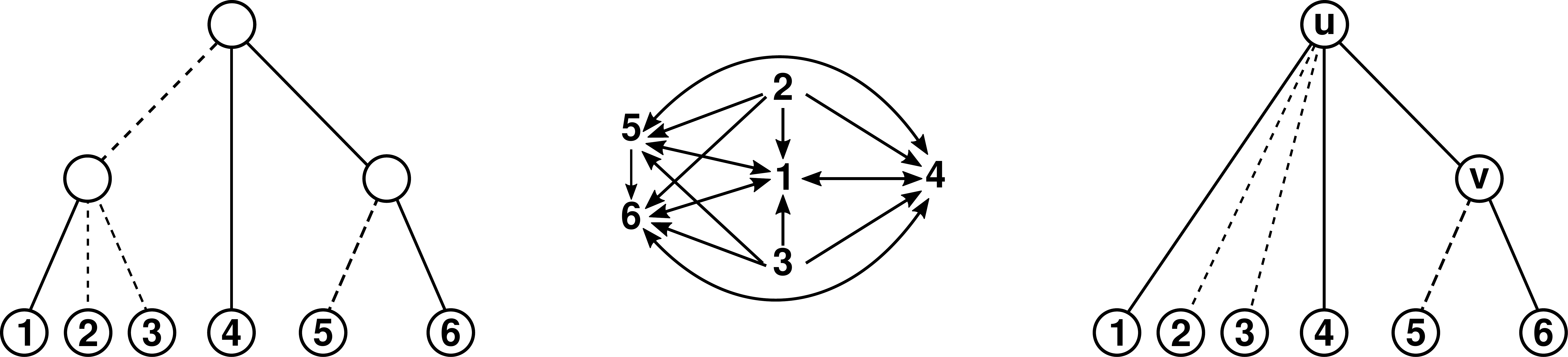}
\end{center}
\caption{An edge-labeled tree $(T,\lambda)$ \emph{(left)}
				 and the resulting relation $\rel_{(T,\lambda)}$ \emph{(middle)}
				 is shown. The unique least-resolved tree $(T^*,\lambda^*)$
					that explains $\rel_{(T,\lambda)}$ 
				 is shown in the right panel. All 1- and 0-edges are highlighted by 
					drawn-through and dashed lines, respectively.}
\label{fig:exmpl}
\end{figure}

\begin{definition}
An edge-labeled 
tree $(T,\lambda)$ on $L$ \emph{explains}
a given relation $\rel$ on $L$, whenever there is a 1-edge on
the path from $\lca(x,y)$ to $y$ if and only if $(x,y)\in \rel$, i.e., 
$\rel=\rel_{(T,\lambda)}$. 

A relation $\rel$  that can be explained by an edge-labeled tree is called \emph{Fitch relation}. 
\end{definition}

The enumeration of all induced subgraphs of size three of a relation $\rel$ is shown in Fig.\ \ref{fig:triangles}:
up to isomorphism there are 16 subgraphs $A_1$-$A_8$, called \emph{allowed triangles}, and
$F_1$-$F_8$, called \emph{forbidden triangles}.
It has been shown by Gei{\ss} et al.\ \cite{Geiss2018}
that Fitch relations can be characterized in terms of such triangles.

\begin{theorem}[{\cite[Thm.\ 2]{Geiss2018}}]
\label{thm:Geiss-char}
A relation $\rel$ is a Fitch relation if and only if 
$\rel$ does not contain one of the forbidden triangles $F_1$-$F_8$ as an induced subgraph
and hence, all induced subgraphs on three vertices are isomorphic to one of the allowed triangles $A_1$-$A_8$. 	
\end{theorem}

$(T,\lambda)$ is \emph{least-resolved} w.r.t.\ a relation $\rel$, if 
$(T,\lambda)$ explains $\rel$ and 
there is no coarsement $T'< T$ 	 and no labeling $\lambda'$ 
such that the edge-labeled tree $(T',\lambda')$ still explains $\rel$. 
Gei{\ss} et al.\ \cite{Geiss2018} characterized such least-resolved trees
 and, even more, showed that they are unique.

\begin{theorem}[{\cite[Thm.\ 1 and Lemma 11]{Geiss2018}}] 
	\label{thm:least-res}
	Let $\rel$ be a Fitch relation  and 
  $(T,\lambda)$ be an edge-labeled tree  that explains $\rel$.  The
  following two statements are equivalent:
  \begin{enumerate}
  \item \label{it:-1} $(T,\lambda)$ is least-resolved 
  w.r.t.\ $\rel$.
  \item\label{it:1}
    (a) Every inner edge  of $(T,\lambda)$ is a 1-edge and \\
    (b) for every inner edge $(\parent(v),v)$ there is an outer 0-edge $(v,x)$ in
        $(T,\lambda)$.
	\end{enumerate}

	Moreover, the least-resolved tree w.r.t.\ the Fitch relation $\rel$ is, up to isomorphism, 
	unique. 
\end{theorem}

It is worth to be mentioned that deciding whether a relation $\rel$ on $L$ is a Fitch relation and, in the positive
case, to construct the unique least-resolved tree $(T,\lambda)$ that explains $\rel$ can be done
in $O(|L|+|\rel|)$ time, cf.\ \cite[Section 6]{Geiss2018}.

\begin{figure}[t]
\begin{center}
  \includegraphics[width=0.85\textwidth]{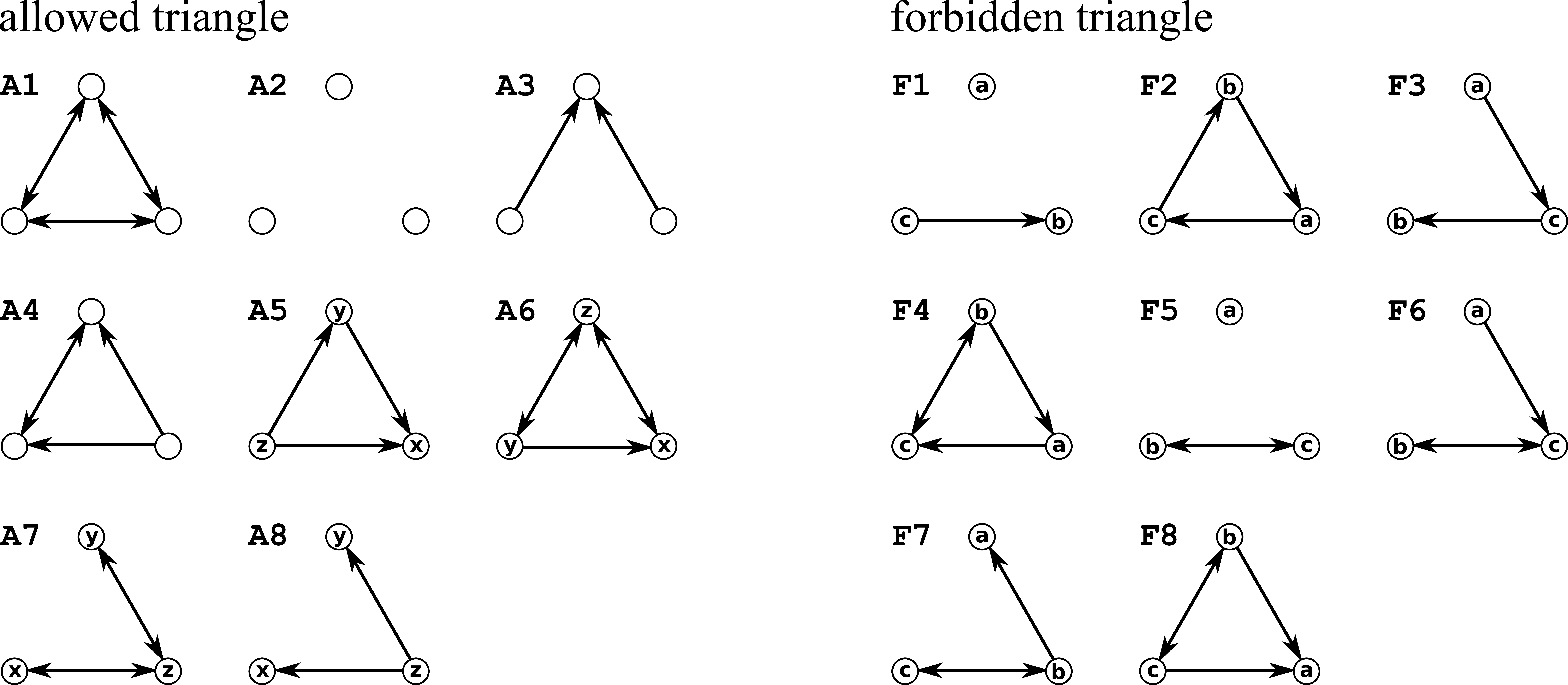}
\end{center}
\caption{ Shown is the graph representation for all
  possible relations $\rel\subseteq \irr{L\times L}$ with $|L|=3$. The relations are grouped into allowed ($A_1-A_8$) and forbidden ($F_1-F_8$) triangles. 
	\emph{The figure is adopted from \cite{Geiss2018}.}
	}
 
\label{fig:triangles}
\end{figure}

\section{Alternative Characterizations}

\subsection{Characterization via Neighborhoods}

Theorem \ref{thm:Geiss-char} provides a characterization in terms of
forbidden and allowed triangles. In what follows, we will present  
a new characterization in terms of neighborhoods. 

\begin{definition}
	Let $\rel$ be a relation on $L$. 
  The \emph{(complementary) neighborhood} $N[y]$ of $y \in L$ w.r.t.\ $\rel$ is defined as follows:
	\[ N[y] \coloneqq \{ x \in L \setminus \{y\} \colon (x,y)\notin \rel\} \cup \{y\}. \]
	Moreover, we define the 
	set of neighborhoods w.r.t.\ $\rel$ as follows:
	\begin{align*}
		\Ns[\rel] &\coloneqq \{ N[y] \colon y \in L\}.
	\end{align*}
\end{definition}

Essentially, $N[y]$ covers $y$ and all incoming neighbors of $y$ in the complement of $\rel$. 
To give some intuition, why we defined  $N[y]$ and 	$\Ns[\rel]$ we refer first to 
Theorem \ref{thm:least-res} and consider the unique least-resolved tree $(T,\lambda)$
that explains the Fitch relation $\rel$. 
Theorem \ref{thm:least-res} implies that each inner edge $(\parent(v),v)$ 
of  $(T,\lambda)$ must be a 1-edge and 
$v$ must be incident to an outer 0-edge $(v,y)$. Now, consider the cluster 
$\Cl{T}(v)$. Since $(v,y)$ is a 0-edge and $y$ is a leaf, we can observe
that for every leaf
$x\in \Cl{T}(v)$ with $x\neq y$ it must hold $(x,y)\notin \rel$. 
Moreover, since $(\parent(v),v)$ is a 1-edge, we have $(x,y)\in \rel$ for all leaves 
$x\notin \Cl{T}(v)$. Hence, the latter two arguments, together with $y\in N[y]$, imply that 
$N[y]$ provides precisely all elements of this particular cluster $\Cl{T}(v)$. 
For an example, consider the least-resolved tree $(T^*,\lambda^*)$ for 
$\rel = \rel_{(T,\lambda)}$ in Fig.\ \ref{fig:exmpl}. 
There is only one inner 1-edge $(u,v)$ in $(T^*,\lambda^*)$. 
The vertex $v$ is incident to the outer 0-edge $(v,5)$. Obviously, $(6,5)\notin \rel$
and $N[5]=\{5,6\}=\Cl{T^*}(v)$. In particular, $\Ns[\rel]$ contains \emph{all} 
non-trivial clusters of $T^*$. 

In what follows, we show that such neighborhoods can be used to characterize
Fitch relations. Essentially, we prove that $\rel$ is a Fitch relation if and only
if  $\Ns[\rel]$ is hierarchy-like and the elements in  $\Ns[\rel]$ are ``well-behaved''
as defined as follows.
\begin{definition} \label{C:D:conditions, HLC, IC, ELC}
	Let $\rel$ be a  relation on $L$. Then, we say that $\rel$ satisfies
	\begin{itemize}
		\item[$\bullet$] the \emph{hierarchy-like-condition (\HC)}, if $\Ns[\rel]$ is hierarchy-like; and 
		\item[$\bullet$] the \emph{inequality-condition (\IC)}, if for every neighborhood $N \in \Ns[\rel]$ and every $y \in N$, we have $|N[y]|\le|N|$.
	\end{itemize}
\end{definition}

To see the intuition behind the inequality-condition, consider the relation 
$\rel = \{(c,b)\}$ on $L=\{a,b,c\}$ which corresponds to the forbidden
subgraph $F_1$ in Fig.\ \ref{fig:triangles}. Hence, $\rel$ is not a Fitch relation. 
In this example, $N[a] = N[c] = \{a,b,c\}$ and $N[b] = \{a,b\}$. Thus, 
$\Ns[\rel] = \{\{a,b,c\},\{a,b\}\}$ is hierarchy-like, 
although $\rel$ cannot be explained by any tree. 
Hence, hierarchy-likeness of $\Ns[\rel]$ is not sufficient to characterize
Fitch-relations. However, for $N\coloneqq N[b] \in \Ns[\rel]$ and $a\in N$
we have $|N[a]|>|N|$ and thus, $\rel$ does not satisfy \IC. 
As we shall see, Fitch relations are characterized by \HC and \IC. 

We start with proving the necessity of \HC.

\begin{lemma} \label{lemma: neighbors are clusters}
	Let $\rel$ be a Fitch relation on $L$ and $(T,\lambda)$ be some edge-labeled
	tree that explains $\rel$. 
	Then, we have  $N[y] \in \Cls(T)$ for all $y \in L$ and 
	$\Ns[\rel] \subseteq\Cls(T)$.
\end{lemma}
\begin{proof}
	Let $\rel$ be a Fitch relation and $(T = (V,E),\lambda)$ be an edge-labeled tree that explains
	$\rel$. Moreover, let $y \in L$ be chosen arbitrarily. By definition,  
	$N[y]$ contains always the vertex $y\in L$ and therefore,
	$|N[y]|\ge 1$.

	If $|N[y]| = 1$, then $N[y] = \{y\}=\Cl{T}(y) \in \Cls(T)$. Now,
	assume that $|N[y]| \ge 2$. Clearly, it holds that 
	$N[y]\subseteq L =\Cl{T}(\rootT{T})$. Therefore, we can choose a vertex $v\in
	V$ with $N[y] \subseteq \Cl{T}(v)$ such that $|\Cl{T}(v)|$ is
	minimal. Moreover, $v$ must be an inner vertex,
	since $|\Cl{T}(v)|\ge |N[y]|\ge 2$. This and $y\in \Cl{T}(v)$ imply that 
there is a vertex $w \in
	V$ with $(v,w)\in E$ and $v \prec w \preceq y$. Since $T$ is phylogenetic and
	due to the minimality of $|\Cl{T}(v)|$, we can conclude that
	$N[y]\nsubseteq \Cl{T}(w)$. The latter implies that there exists a
	vertex $x \in N[y]\setminus \Cl{T}(w) \subseteq \Cl{T}(v)$. Since $x
	\notin \Cl{T}(w)$, the leaf $x$ and the vertex $w$ are incomparable in $T$.
	The latter, together with $(v,w)\in E$ and $v \prec x$, implies that $v =
	\lca(x,w)$. Moreover, $w \preceq y$ immediately implies that $v =
	\lca(x,y)$. Since $x \in N[y]$ implies $(x,y)\notin \rel$ and
	since $(T,\lambda)$ explains $\rel$, we can conclude that there is \emph{no}
	$1$-edge on the path from $v=\lca(x,y)$ to $y$.
	
	Note, we have chosen $v$ such that $N[y] \subseteq \Cl{T}(v)$.
	Assume for contradiction that $N[y]\neq \Cl{T}(v)$. Hence, there is
	a vertex $x' \in \Cl{T}(v) \setminus N[y]$. Since $x',y\in
	\Cl{T}(v)$, we obtain $v \preceq \lca(x',y)$. Since $x' \notin N[y]$, we have $(x',y)\in \rel$. 
	This, together with the fact that	$(T,\lambda)$ explains $\rel$, implies that there is 
  a $1$-edge on the path
	from $\lca(x',y)$ to $y$. However, since $v \preceq \lca(x',y) \preceq y$,
	this $1$-edge is also contained on the path from $v$ to $y$; a contradiction. 
	Hence, $N[y] = \Cl{T}(v) \in \Cls(T)$. 
	
	Finally, since for all  $N[y] \in \Ns[\rel]$ we have 
	$N[y]\in \Cls(T)$, we can conclude that 
 	$\Ns[\rel]\subseteq \Cls(T)$. 
  \end{proof}

Using Lemma \ref{lemma: neighbors are clusters} and the fact that (every subset
of) the cluster-set $\Cls(T)$ of a phylogenetic tree $T$ is always
hierarchy-like, we immediately obtain the following
\begin{corollary}
	Every Fitch relation $\rel$ satisfies the hierarchy-like-condition (\HC). 
\end{corollary}

To show the necessity of \IC, we start with following lemma.

\begin{lemma} \label{lemma: properties of neighborhoods in trees}
	Let $\rel$ be a Fitch relation on $L$ and $(T,\lambda)$ be some edge-labeled tree that explains $\rel$. 	
	Moreover, let $y \in L$ be a leaf and chose $v \in V(T)$ such that  $N[y]	= \Cl{T}(v)$.
	Then,
	the following two properties are satisfied:
	\begin{enumerate}
		\item[(a)] There is no $1$-edge on the path from $v$ to $y$.
		\item[(b)] If $N[y] \ne L$, then $(\parent(v),v)$ is a $1$-edge.
	\end{enumerate}
\end{lemma}
\begin{proof} 	
	Let $\rel$ be a Fitch relation on $L$,  $(T = (V,E),\lambda)$ be an edge-labeled tree that explains
	$\rel$ and $y \in L$ be an arbitrary leaf. 
	Lemma \ref{lemma: neighbors are clusters} implies that we can choose a  vertex $v\in
	V$ with $N[y] = \Cl{T}(v)$. 
	Clearly, if $y=v$, then the lemma is trivially satisfied. Hence, assume that $v\neq y$. 

	Obviously, $v$ is an ancestor of $\lca(x,y)$ for every $x\in \Cl{T}(v) =
	N[y]$. Moreover, since $T$ is phylogenetic, there is a vertex $z \in
	\Cl{T}(v)=N[y]$ with $\lca(z,y)=v$. Since $z\in N[y]$ we have $(z,y)\notin \rel$. 
	This and the fact that $(T,\lambda)$ explains $\rel$ imply
	that the path from $\lca(z,y)=v$ to $y$ does not
	contain a $1$-edge, which
	proves Property (a).

	We continue with showing Property (b). Since $N[y] = \Cl{T}(v)$ and 
	$N[y]\neq L$, we can conclude that $v \neq \rootT{T}$.
	Since $v \ne \rootT{T}$ and $T$ is phylogenetic, there must be a parent
	$\parent(v)$ of $v$ and a vertex $x' \in \Cl{T}(\parent(v)) \setminus
	\Cl{T}(v)$. Hence, $x' \notin \Cl{T}(v)= N[y]$ and thus, 
	 $(x',y) \in \rel$. By the choice of $x'$, we have
	$\lca(x',y)=\parent(v)$. Since $(T,\lambda)$ explains $\rel$, we can conclude
	that there is a $1$-edge along the path from
	$\lca(x',y)=\parent(v)$ to $y$. Moreover, Property (a) implies that there is
	\emph{no} $1$-edge along the path from $v$ to $y$. Taken the latter arguments
	together, $(\parent(v),v)$ must be a $1$-edge.
  \end{proof}

\begin{lemma}
	Every Fitch relation $\rel$ satisfies the inequality-condition (\IC). 
\end{lemma}
\begin{proof}
	Let $\rel$ (on $L$) be a Fitch relation and $(T=(V,E),\lambda)$ be an edge-labeled tree that
	explains $\rel$. In order to prove the statement we need to show that
	$|N[y]|\le |N|$ is satisfied for every $y' \in L$ and every $y \in N\coloneqq N[y']$. 
	
	Let $y' \in L$ and let $y \in N \coloneqq N[y']$. 
	Clearly, if $y=y'$, then $|N[y]| =  |N|$. Thus, assume that $y\neq y'$.
	Lemma \ref{lemma: neighbors are clusters} implies that $N \in	\Cls(T)$. 
	Hence, there is a vertex $v \in V$ such that $\Cl{T}(v)= N$. Since
	$y \in N = \Cl{T}(v)$, we can conclude that $v$ is an ancestor of $y$ in $T$.

	In what follows, we distinguish two mutually exclusive cases, either $v =
	\rootT{T}$ or $v \ne \rootT{T}$. If $v = \rootT{T}$, then $N = L$. Hence,
	$N[y] \subseteq L = N$ and therefore $|N[y]| \le |N|$.
	Now, assume that $v \ne \rootT{T}$ and therefore, $N = \Cl{T}(v)\neq L$.
	Again, Lemma \ref{lemma: neighbors are clusters} implies that 
	$N[y] \in \Cls(T)$ is a cluster. Hence, there is a
	vertex $w \in \Cls(T)$ such that $N[y] = \Cl{T}(w)$. Note that we
	have by definition $y \in N[y] = \Cl{T}(w)$. This, together with $y
	\in N = \Cl{T}(v)$, implies that $y \in \Cl{T}(v)\cap \Cl{T}(w)$. 
	Since $\Cl{T}(v)\cap \Cl{T}(w)\neq \emptyset$ and $\Cls(T)$ is a hierarchy, we
	can conclude that $\Cl{T}(v)\cap \Cl{T}(w) \in \{\Cl{T}(v),\Cl{T}(w)\}$.
	The latter 	implies 
	either $v\preceq w$ or $w \prec v$.
	
	First, assume that $w \prec v$.
	Since we have $\Cl{T}(v) =N \ne L$, we can apply Lemma \ref{lemma:
	properties of neighborhoods in trees} (b) to conclude that 
  $\lambda(\parent(v),v) =1$. 
	 This, together with $w \prec v$, implies that there is
	a $1$-edge on the path from $w$ to $y$;
	a contradiction to Lemma \ref{lemma: properties	of neighborhoods in trees} 
	(a) applied on $N[y] = \Cl{T}(w)$.
	Hence, $w \prec v$ is not possible, and therefore, it must hold that $v\preceq w$. 
	Thus, we have $N[y]=\Cl{T}(w)\subseteq \Cl{T}(v) = N$, and therefore, we obtain
	$|N[y]| \le |N|$. 
  \end{proof}

To show that \HC and \IC are also sufficient for Fitch relations, we first 
define an edge-labeled tree $\To{\rel} = (T, \lambda)$ based on $\Ns[\rel]$
and prove that $\To{\rel}$ explains $\rel$.   

\begin{definition}
	Let $\rel$ be a relation on $L$ that satisfies $\HC$. 
	The edge-labeled tree $\To{\rel} = (T,\lambda)$  has the following cluster set
	\[\Cls(T) = \Ns[\rel] \cup \big\{L\big\} \cup \big\{ \{x\} \colon x \in L \big\} \]
	and 	each edge $(\parent(v),v)$ of $T$ obtains the label
	\[	\lambda(\parent(v),v) =1 \iff  \Cl{T}(v)\in \Ns[\rel]\text{, i.e., there is some }  y\in L 
														 \text{ with } N[y]=\Cl{T}(v).
	\] 
	\label{def:tree-for-rel}
\end{definition}

We emphasize that all non-trivial clusters of $\To{\rel}$ are provided by the elements in $\Ns[\rel]$. 
Now, we are in the position to show that \HC and \IC are sufficient for Fitch relations.

\begin{lemma}
	\label{lem:C:T:Sufficiency}
		Let  $\rel$ be a relation on $L$  	that satisfies $\HC$ and $\IC$. 
	 Then, 	$\To{\rel}$ is well-defined. Moreover, $\To{\rel}$ explains $\rel$, and thus, 	
	$\rel$ is a Fitch relation.	
\end{lemma}
\begin{proof}
	Since $\rel$ satisfies \HC, the set $\Ns[\rel]$ is hierarchy-like. Hence, 
	$\Cls(T)$ as in Def.\ \ref{def:tree-for-rel} is indeed a hierarchy. This, together 
	with Theorem \ref{thm:hierarchy-tree}, implies that 
	$T$ is well-defined. In particular, $T$ is a phylogenetic tree on $L$.
	The edge-labeling $\lambda: E(T) \to \{0,1\}$ in Def.\ \ref{def:tree-for-rel} is
	only based on the existence of some $y\in L$ with $N[y]=\Cl{T}(v)$, and thus, is well-defined as well.
	In summary, the edge-labeled tree 	$\To{\rel}=(T=(V,E),\lambda)$ is well-defined.

	Now, we prove
	that $\To{\rel}$ explains $\rel$. To this end, we will show that 
	$(x,y)\in \rel$ if and only if there is a 1-edge on the path from 
	$\lca(x,y)$ to $y$ in $\To{\rel}$.

	First, suppose that $(x,y)\in \rel$. 
	By definition $N[y]\in \Ns[\rel]$ and, by construction of 
	$T$, we have	$N[y]\in \Ns[\rel] \subseteq \Cls(T)$.
	Thus, there is some vertex $v\in V$ with $N[y] = \Cl{T}(v)$.
	Since $(x,y)\in \rel$, we have 	$x \notin N[y]$ and therefore $N[y] \ne L$.
	The latter implies that $v\neq \rootT{T}$, and hence, $\parent(v)$ exists.
	By construction of $\lambda$, we have $\lambda(\parent(v),v) = 1$. 
	Now,  $x \notin N[y] = \Cl{T}(v)$ and  $y\in N[y] = \Cl{T}(v)$ 
	imply
	that $\lca(x,y) \preceq \parent(v)$. Hence, 
	the 1-edge $(\parent(v),v)$ is located on the path from 
	$\lca(x,y)$ to $y$. 

	Conversely, assume that 
	 $x,y\in L$ are distinct vertices
	such that the path from $\lca(x,y)$ to $y$ contains 
	a 1-edge $(\parent(v),v)$. 
	Hence, by construction of $\lambda$, 
	we have $\Cl{T}(v) \in \Ns[\rel]$.

	We continue to show that $N[y] \subseteq \Cl{T}(v)$. 
	Since $v$ is located on the path from $\lca(x,y)$ to $y$, we have $v\preceq y$, and thus, $y\in \Cl{T}(v)$.
	Since $\rel$ satisfies \IC and $y\in N\coloneqq \Cl{T}(v) \in \Ns[\rel]$ it must hold that
	$|N[y]|\leq |N| = |\Cl{T}(v)|$. 
	Moreover, since $\rel$ satisfies \HC, the set $\Ns[\rel]$ is hierarchy-like. 
	Since $\Cl{T}(v), N[y] \in \Ns[\rel]$ and $y\in \Cl{T}(v)\cap  N[y]$, 
	we have either   $\Cl{T}(v)\subsetneq N[y]$ or $N[y] \subseteq \Cl{T}(v)$.
  However, $|N[y]|\leq | \Cl{T}(v)|$ immediately implies that 
	$N[y] \subseteq \Cl{T}(v)$ must hold. 

	Furthermore, since 
	$(\parent(v),v)$ is located on the path from $\lca(x,y)$ to $y$, 
	we can conclude that $x\notin \Cl{T}(v)$. This and $N[y] \subseteq \Cl{T}(v)$
	imply
	that $x\notin N[y]$. Hence, the definition of  $N[y]$ implies that $(x,y)\in \rel$. 

	To summarize, for any two distinct vertices $x,y\in L$ we have
	$(x,y) \in \rel$ if and only if there is a 1-edge on the path from 
	$\lca(x,y)$ to $y$ in $\To{\rel}$. Therefore, $\To{\rel}$ explains $\rel$. 
	Hence, $\rel$ is a Fitch relation. 
  \end{proof}

As a consequence of the results above, we obtain the following new characterization. 
\begin{theorem}
	\label{thm:new-characterization}
	A relation $\rel$ is a Fitch relation if and only if $\rel$ satisfies $\HC$ and $\IC$. 
\end{theorem}

For the sake of completeness we show that $\To{\rel}$ is 
least-resolved w.r.t.\ a Fitch relation $\rel$.

\begin{proposition}
	The tree $\To{\rel}$ is the unique least-resolved tree that explains 	
	the Fitch relation $\rel$. 	
\end{proposition}
\begin{proof}
	By construction of $\Cls(T)$ all non-trivial clusters
	of $T$ are provided by $\Ns[\rel]$. Hence, for 
	each non-trivial cluster $\Cl{T}(v)$ of $T$ there is
	a vertex $y\in L$ with $N[y]=\Cl{T}(v)$. 
	The latter and the construction of $\lambda$  
	imply
	that all inner edges of  $\To{\rel}$ are 1-edges.
	Moreover, since $\To{\rel}$ explains $\rel$, we can apply
	Property (a) of Lemma \ref{lemma: properties of neighborhoods in trees}
	and conclude that there is no $1$-edge on the path from $v$ to $y$.
	However, since all inner edges of $\To{\rel}$ are 1-edges, 
	the path from $v$ to $y$ is simply the edge $(v,y)$. 
	Hence, $(v,y)$ is an outer 0-edge. Therefore, each inner vertex $v\neq \rootT{T}$
	is incident to an outer 0-edge. Now, we can apply 
	Theorem \ref{thm:least-res} to conclude that 
	$\To{\rel}$ is the unique least-resolved tree that explains $\rel$.
 \end{proof}

\subsection{Characterization via Three-Vertex Subrelations}

Theorem \ref{thm:Geiss-char} provides a characterization in terms of
forbidden and allowed triangles. All allowed triangles $\Delta\in \{A_1,A_2,\dots, A_8\}$ share a common
property, namely if $\Delta$ contains an edge  $(c,b)$ but no edge $(a,b)$, 
then $(c,a)$ must be an edge in $\Delta$ and either 
both $(a,c),(b,c)$ are an edge in $\Delta$ or neither of them is an edge in $\Delta$. 
This, in fact, characterizes allowed triangles as shown in the next lemma. 

\begin{lemma} \label{M:L:charact. valid triangles}
	Let $\rel$ be a relation on $L$. Then, the following two conditions
	are equivalent:
	\begin{enumerate}
		\item For every $\{x,y,z\} \in {L\choose 3}$ the subgraph of $\rel$
					that is induced by $\{x,y,z\}$ is isomorphic to one of the 
					allowed triangles $A_1,A_2, \dots ,A_8$.
		\item for every $\{a,b,c\} \in {L\choose 3}$ 
					with $(c,b)\in \rel$ and $(a,b) \notin\rel$, 
					we have $(c,a)\in\rel$ and either $(a,c),(b,c)\in \rel$ or $(a,c),(b,c)\notin \rel$ .
	\end{enumerate}
\end{lemma}
\begin{proof}
	Let $\rel$ be a relation on $L$.
	Recap that $\rel[\{u,v,w\}]$ denotes the subgraph of $\rel$ that is induced by the vertices $u,v,w\in L$. 
	In this proof, we use the labels of the vertices for the allowed and forbidden triangles 
	as shown in Fig.\ \ref{fig:exmpl}.
	
	Assume for contraposition that Condition (2) is not satisfied.
	Hence, there is a subset $\{a,b,c\} \in {L\choose 3}$ 
	such that $(c,b)\in \rel$ and $(a,b) \notin \rel$, but 
	\begin{itemize}
	\item[(i)] $(c,a) \notin \rel$ or
	\item[(ii)] either $(a,c)\in \rel $ and $ (b,c)\notin \rel$ or $(a,c)\notin \rel$ and $ (b,c)\in \rel$.
	\end{itemize}
	Put $\Delta\coloneqq \rel[\{a,b,c\}]$. Since  $(c,b)\in \rel$ and $(a,b) \notin \rel$, 
	the in-degree of $b$ in $\Delta$ must be one. Since none of the graphs $A_1,A_2,A_3$ and $A_4$
	contains a vertex with in-degree one, $\Delta$ cannot be isomorphic to $A_1,A_2,A_3$ or $A_4$.
	
	Assume for contradiction that $\Delta$ is isomorphic to $A_5$. 
	Since $A_5$	contains only one vertex with in-degree one  namely $y$, we obtain
	$b=y$. 
	Moreover, $(c,b)\in \rel$ and $(a,b)\notin \rel$ imply
	$x=a$ and $z=c$.
	Since $(z,x)\in E(A_5)$, we have $(c,a)\in E(\Delta)\subseteq \rel$. Hence, Case (i) cannot be satisfied. 
	Now, $(y,z),(x,z)\notin E(A_5)$ implies that $(b,c),(a,c)\notin E(\Delta)\subseteq \rel$, and therefore, Case (ii) cannot be satisfied. 
	The latter two arguments lead to a contradiction, since at least one of the  Cases (i) or (ii) has to be satisfied.
	Thus,  $\Delta$ cannot be isomorphic to $A_5$.

	Assume for contradiction that $\Delta$ is isomorphic to $A_6$. 
	Analogously as in the case for $A_5$, we have $b=y$, $x=a$ and $z=c$
	and, by similar arguments, we obtain that $\Delta$ cannot be isomorphic to $A_6$.

	Assume for contradiction that $\Delta$ is isomorphic to $A_7$ or $A_8$.
	For both graphs, $(c,b)\in  \rel$ and $(a,b) \notin \rel$ imply 
	$z=c$.
	Since $b$ has in-degree one in $\Delta$, we observe
	that $b\in\{x,y\}$ for both graphs. 
	Due to symmetry, we can w.l.o.g.\ choose $a=x$ and $b=y$. 
	Since $(z,x) \in E(A_7)\cap E(A_8)$, we have $(c,a) \in E(\Delta)\subseteq \rel$. Hence, Case (i) $(c,a) \notin \rel$ is not possible.
	If $\Delta$ is isomorphic to $A_7$, then we have $(a,c),(b,c)\in E(\Delta)\subseteq \rel$, since $(x,z),(y,z) \in E(A_7)$. 
	Now, if $\Delta$ is isomorphic to $A_8$, then we have $(a,c),(b,c)\notin E(\Delta)\subseteq \rel$, since $(x,z),(y,z) \notin E(A_8)$. 
	Again, 	the latter arguments lead to a contradiction, since at least one of the  Cases (i) or (ii) has to be satisfied.
  Thus, $\Delta$ cannot be isomorphic to $A_7$ or $A_8$.

	Thus, if Condition (2) is not satisfied, then $\Delta$ cannot be isomorphic to 
	one of $A_1,\dots,A_8$, and therefore, Condition (1) is not satisfied.

	By contraposition, assume that Condition (1) is not satisfied. Hence, there is
	a subset $\{a,b,c\} \in {L\choose 3}$, such that the induced subgraph
	$\rel[\{a,b,c\}]$ is not isomorphic to $A_1,\ldots,A_7$ or
	$A_8$; and therefore, $\rel[\{a,b,c\}]$ is isomorphic to one of
	$F_1,\ldots,F_8$. First, we observe that for every forbidden triangle $F_1$ to
	$F_8$ the vertices are labeled such that $(c,b)\in \rel$ and $(a,b)\notin \rel$, 
	see Figure \ref{fig:exmpl}. We also observe that
	$(c,a) \notin \rel$ for $F_1,F_2,\ldots,F_6$ and $F_7$, and that
	$(a,c)\notin \rel$ and  $(b,c)\in \rel$ for $F_8$.

	Either way, we have found a subset $\{a,b,c\} \in {L\choose 3}$
	such that $(c,b)\in \rel$ and
	$(a,b) \notin \rel$, but $(c,a) \notin \rel$ or
	$(a,c)\notin \rel$ and  $(b,c)\in \rel$.
	 Thus, Condition (2) is not satisfied.	
  \end{proof}

Based on  Theorem \ref{thm:Geiss-char} and Lemma \ref{M:L:charact. valid triangles} 
we obtain the following new characterization of Fitch relations. 
\begin{theorem}
	A relation $\rel$ on $L$ is a
	Fitch relation if and only if for every subset $\{a,b,c\} \in {L\choose 3}$
 	with $(c,b)\in \rel$ and $(a,b) \notin\rel$, 
	we have $(c,a)\in\rel$ and either $(a,c),(b,c)\in \rel$ or $(a,c),(b,c)\notin \rel$ .
\end{theorem}

\section{An Alternative Proof of Theorem \ref{thm:Geiss-char} }

The key idea of the proof of Theorem \ref{thm:Geiss-char} in \cite[Section 5]{Geiss2018},
which proceeds by induction on the number of leaves,
is to consider the superposition of trees explaining two induced subrelations $\rel_1,\rel_2$, each
of which is obtained from $\rel$ by removing a single vertex from $L$.
This proof, however, is quite involved and very technical, and includes plenty of 
case studies. The characterization of Fitch relations in terms of \HC and \IC 
allows us to establish a significantly shorter and simpler proof of  Theorem \ref{thm:Geiss-char}, 
which we present here. We emphasize that this new proof is solely based on 
Theorem \ref{thm:new-characterization} and Lemma \ref{M:L:charact. valid triangles}.

\begin{proof}[Alternative proof of Theorem \ref{thm:Geiss-char}]

	We omit the ``if-direction'' of the	proof of Theorem \ref{thm:Geiss-char}, 
	since the ``if-direction''  in  \cite{Geiss2018} is fairly
	simple and straightforward to obtain. It
	is based on full enumeration of all 16 edge-labeled trees on three vertices, which eventually shows that
	none of the forbidden triangles can be explained by a tree. 

  For the ``only-if-direction'', let $\rel$ be a relation on $L$ and assume, for contraposition, that 
	$\rel$ is not a Fitch relation. 
	Thus, Theorem \ref{thm:new-characterization} implies that $\rel$ does not satisfy \HC or \IC.
	
	Assume first that $\rel$ does not satisfy \IC. Hence, there is a neighborhood $N\in \Ns[\rel]$ 
	and a vertex $a \in N$ such that $|N[a]|>|N|$. Since $N \in \Ns[\rel]$, there is a vertex $b \in L$ 
	such that $N[b]=N$. Note that $a,b\in N[b]$. The latter, together with 
	$|N[a]|>|N|=|N[b]|$, implies that $a\ne b$ and the existence of a vertex $c \in N[a]\setminus N[b]$. 
	In particular, we have $c\neq a$ and $c\neq b$ and hence $\{a,b,c\} \in {L\choose 3}$. 
		Since $c\notin N[b]$, it must hold that $(c,b) \in \rel$. 
	Since $a \in N[b]$, it must hold that $(a,b) \notin\rel$.
	Since $c \in N[a]$, it must hold that $(c,a) \notin\rel$.	
	Therefore,  Lemma \ref{M:L:charact. valid triangles}(2) is not satisfied, which implies that 
  Lemma \ref{M:L:charact. valid triangles}(1) cannot be satisfied. 
	Therefore, $\rel$ must contain one of the forbidden triangles $F_1,\dots,F_8$.

	Now assume that $\rel$ does not satisfy \HC and thus, $\Ns[\rel]$ is not hierarchy-like.
	Hence, there are two neighborhoods $N,N' \in \Ns[\rel]$, such that $N\cap N' \notin \{\emptyset,N,N'\}$.
  The latter implies, in particular, $N\neq N'$. 
	This and $N,N' \in \Ns[\rel]$ imply that 
	there are two distinct vertices $y,y' \in L$ such that $N[y]=N$ and $N[y']=N'$. 
	Moreover,  $N\cap N' \notin \{\emptyset,N,N'\}$ implies that $N$ and $N'$ are not disjoint. 
	Since $y \in N[y]$ and $y' \in N[y']$, there are two mutually exclusive cases that need to be examined:
		\begin{align*}
			\text{\emph{(a)}} & \text{ none of } y \text{ and } y' \text{ is contained in } N[y']\cap N[y] \text{, and} \\
			\text{\emph{(b)}} & \text{ at least one of } y \text{ and } y' \text{ is contained in } N[y']\cap N[y]. 
		\end{align*}
	\begin{description}
		\item{\emph{Case (a)}:} This case is equivalent to 
			$y \in N[y]\setminus N[y']$ and $y' \in N[y']\setminus N[y]$. 
			Since $N[y']\cap N[y]\neq \emptyset$, there is a vertex $x
			\in N[y]\cap N[y']$ with $x\neq y,y'$. Thus, 
			$x,y$ and $y'$ are pairwise distinct. 
			Since $y \notin N[y']$,  we have $(y,y') \in \rel$.	
			Since $y' \notin N[y]$, we have $(y',y) \in \rel$.  
			Since $x \in N[y']$, we have $(x,y') \notin\rel$, and since $x \in N[y]$, we have $(x,y) \notin \rel$. 
			Now, put $a\coloneqq x$, $b\coloneqq y'$ and $c\coloneqq y$. Hence, we
			have found a subset  $\{a=x,b=y',c=y\} \in {L\choose 3}$ such that 
			$(c,b) \in \rel$, $(a,b) \notin\rel$ and 
			$(a,c)\notin\rel$ and $(b,c)\in \rel$. Thus, Condition (2) of Lemma
			\ref{M:L:charact. valid triangles} is not satisfied; and therefore,
			Condition (1) of Lemma \ref{M:L:charact. valid triangles} is not
			satisfied. Hence, $\rel$ contains a forbidden triangle.
			\smallskip

		\item{\emph{Case (b)}:} This case is equivalent to $y \in N[y]\cap N[y']$ or $y' \in N[y]\cap N[y']$.
			We can assume w.l.o.g.\ that $y \in N[y]\cap N[y']$. Since $N[y]\cap
			N[y'] \notin \{\emptyset,N[y], N[y']\}$, there is an $x \in N[y]\setminus
			N[y']$. The latter, together with $y,y' \in N[y']$ and $y\ne y'$, implies
			that $x,y$ and $y'$ are pairwise distinct. 
			Since $x \notin N[y']$, we have $(x,y')\in \rel$. 
			Since $y \in N[y']$, we have $(y,y') \notin \rel$. 
			Since $x \in N[y]$, we have $(x,y) \notin \rel$. 
			Now, put
			$a\coloneqq y$, $b\coloneqq y'$ and $c\coloneqq x$. 
			 Hence, we
			have found a subset  $\{a=y,b=y',c=x\} \in {L\choose 3}$ such that 
			$(c,b)\in\rel$ and $(a,b) \notin \rel$, but $(c,a)\notin\rel$. 
			Thus, Condition (2) of Lemma \ref{M:L:charact. valid
			triangles} is not satisfied; and therefore, Condition (1) of Lemma
			\ref{M:L:charact. valid triangles} is not satisfied. 
			Hence, $\rel$ contains a forbidden triangle. 
	\end{description}

	In both Cases (a) and (b), the relation $\rel$ contains a forbidden triangle, 
	which proves the ``only-if-direction''. 
  \end{proof}

\section{Summary}
In this contribution, we gave two novel characterizations of Fitch relations
$\rel$. One characterization is based on the neighborhoods $N[y]$ of vertices
$y\in L$ that comprises vertex $y$ and all vertices $x\in L$ with $(x,y)\notin
\rel$. We have shown, that `well-behaved'' collections $\Ns[\rel]$ 
of such neighborhoods (i.e., they satisfy \HC and \IC)  
characterize Fitch relations. Furthermore, the tree
$\To{\rel} = (T,\lambda)$ with hierarchy $\Cls(T)$ consisting of all members of
$\Ns[\rel]$ together with the (possibly additional) sets $L$ and $\{x\}$ with
$x\in L$ and a well-defined labeling $\lambda$ explains the Fitch relation
$\rel$. In particular, $\To{\rel}$ is the unique least-resolved tree for $\rel$.
The second characterization is based on three-vertex induced subrelations and
the observation that allowed triangles share a common simple property, namely,
if $\rel$ contains $(c, b)$ but not $(a, b)$, then $(c, a)\in \rel$ and either
both $(a, c), (b, c)$ or none of them are contained in $\rel$. These results are
used to establish a simpler and significantly shorter proof of the
characterization theorem provided by Gei{\ss} et al.\ \cite[Thm.\ 2]{Geiss2018}. 

\section*{Acknowledgements}
We are grateful to Manuela Gei{\ss} and Peter F.\ Stadler for all the fruitful
and interesting discussions. 
Moreover, we thank Carmen Bruckmann and Annemarie Luise K{\"u}hn for their constructive comments and suggestions
that helped to improve the paper.

\bibliographystyle{abbrv}
\bibliography{literature}

\end{document}